\documentclass[]{article}
\usepackage{amssymb,enumerate,amsmath,amsthm,mathrsfs,amsfonts,graphics,graphicx,color,tikz,pinlabel}
\usepackage[all]{xy}
\usepackage{tikz-cd}
\usepackage{float}
\usepackage{xcolor}
\usepackage[a4paper, total={6in, 8in}]{geometry}
\theoremstyle{plain}

\title{Strongly Dense Representations of Hyperbolic 3-Manifold Groups}
\author{Ricky Lee}
\date{}

\begin{document}
\maketitle

\newtheorem{thm}{Theorem}[section]
\newtheorem{lemma}[thm]{Lemma}
\newtheorem{prop}[thm]{Proposition}
\newtheorem{cor}[thm]{Corollary}
\newtheorem{defn}[thm]{Definition}
\newtheorem{examp}[thm]{Example}
\newtheorem{conj}[thm]{Conjecture}
\newtheorem{question}{Question}
\newtheorem*{rmk}{Remark}
\newtheorem{Claim}{Claim}
\begin{abstract}
	We provide the first examples of strongly dense representations of a hyperbolic 3-manifold group into $SL(4,\mathbb{R})$ and $SU(3,1)$ i.e. representations where every pair of non-commuting elements has Zariski dense image. Our examples are holonomy representations arising from projective deformations of its hyperbolic structure. As a Corollary, we get that $SL(4,\mathbb{R})$ has non-Hitchin strongly dense surface subgroups.
\end{abstract}

\section{Introduction}

Call a subgroup $\mathcal{H}$ of a semi-simple algebraic group $\mathcal{G}$ \emph{strongly dense} if all pairs of non-commuting elements in $\mathcal{H}$ generate a Zariski dense subgroup of $\mathcal{G}$. Strong density was introduced by Breuillard, Green, Guralnick, and Tao in \cite{Tao}, where they establish the existence of strongly dense free subgroups in many different algebraic groups. But given an algebraic group, it remains a challenging problem to determine which Zariski dense subgroups contain strongly dense free subgroups (see Problem 1 in \cite{Tao} and Question 7.3 in \cite{BIRS}).

Recently there have been advances in the understanding of strongly dense subgroups of $SL(n,\mathbb{R})$ coming from hyperbolic geometry. Say a representation is strongly dense if it has strongly dense image. In \cite{LR}, Long and Reid show that most Hitchin representations of the 344 triangle group into $SL(3,\mathbb{R})$ are strongly dense. Later in \cite{LRW}, Long, Reid, and Wolff extend this result to show that for any closed hyperbolic surface $\Sigma_g$, most Hitchin representations of $\pi_1(\Sigma_g)$ into $SL(n,\mathbb{R})$ are strongly dense. In \cite{BGL}, Breuillard, Guralnick, and Larsen use a similar strategy to establish the existence of a larger class of groups that have strongly dense representations into various algebraic groups.

Given these recent developments, it is natural to wonder which groups can actually occur as strongly dense subgroups of a given semisimple algebraic group. In this paper we show that there are strongly dense representations of closed hyperbolic 3-manifold groups into $SL(4,\mathbb{R})$ and $SU(3,1)$. To the author's knowledge, these are the first examples of strongly dense representations of hyperbolic 3-manifold groups. 

In order to discuss our results more carefully, we review some background. Recall that for $n\geq 3$, if a closed $n$-manifold admits a complete hyperbolic structure then this structure is unique by Mostow rigidity. However, one can regard hyperbolic structure as an instance of real projective structure, and in the projective setting there is no notion of Mostow rigidity. For some hyperbolic 3-manifolds $M$, the holonomy representation $\pi_1(M)\rightarrow SO^+(3,1)$ from the complete hyperbolic structure can be continuously deformed into representations $\pi_1(M)\rightarrow SL(4,\mathbb{R})$ that are holonomies of convex projective structures on $M$. See for instance \cite{Ballas,Computing,Flexing}.

In this paper we focus on the closed hyperbolic manifold known as Vol3, whose deformation theory has been extensively studied in \cite{Computing,Flexing,Thin3}. Vol3 is a closed hyperbolic 3-manifold. It has been shown via arithmetic methods that $\pi_1(Vol3)$ contains no non-elementary Fuchsian subgroups (see Theorem 9.5.5 of \cite{Arithmetic}). Its fundamental group has presentation 
\[	\langle a,b | aabbABAbb, aBaBabaaab \rangle,	\] 
where $A=a^{-1}$ and $B=b^{-1}$. 

A one complex parameter family of representations of Vol3 is given in section 2.4 of \cite{Flexing}. In order to get a family of representations with simpler expressions for the matrix entries, we instead take an orbifold that is four-fold covered by Vol3. The orbifold is obtained from an order 4 rotational symmetry of Vol3 that fixes the axis of $aBaba$ and induces an automorphism of $\pi_1(Vol3)$ such that $a\mapsto BA$ and $b\mapsto aba$. See remark 7.2 in \cite{Computing} for more details. Let $\Gamma_3$ denote the resulting orbifold fundamental group. Then $\Gamma_3$ is generated by two elements $u$ and $c$, where $u^4=1$ and $c=u^2a$. Note that $\pi_1(Vol3)$ can be recovered from $a=u^2c$ and $b=(aua)^{-1}u$. 

We take the one complex parameter family of representations of $\Gamma_3$ found in \cite{Thin3}. The representation is given by 
\[
\rho_{v}(u)=
\begin{bmatrix}
1 & 0 & 0 & 0 \\
0 & 1 & 0 & 0 \\
0 & 0 & \sqrt{\frac{v^2-4}{v^2+8}} & 1\\
0 & 0 & -\frac{2(v^2+2)}{v^2+8} & -\sqrt{\frac{v^2-4}{v^2+8}} 
\end{bmatrix}
\]
and 
\[
\rho_{v}(u)=
\begin{bmatrix}
(v+\sqrt{v^2+8})/4 & 0 & (4-v^2-v\sqrt{v^2+8})/8 & 0 \\
0 & (v-\sqrt{v^2+8})/4 & 0 & (-4+v^2-v\sqrt{v^2+8})/8 \\
1 & 0 & (-v-\sqrt{v^2+8})/4 & 0\\
0 & -1 & 0 & (-v+\sqrt{v^2+8})/4 
\end{bmatrix}.
\]

At the specialization $v=2$, $\rho_2:\Gamma_3\rightarrow SO^+(3,1)$ is the discrete faithful representation coming from the hyperbolic structure. As shown in \cite{Thin3}, for all $v\in (2,\infty)$, deep work of Koszul \cite{Kozul} and Benoist \cite{Benoist1, Benoist2} can be used to show that the representations $\rho_v$ are holonomies of strictly convex projective structures. All such $\rho_v$ are discrete, faithful, and \emph{biproximal}. That is, for every infinite order $\gamma\in \Gamma_3$, $\rho_v(\gamma)$ has a unique eigenvalue of strictly largest modulus, and a unique eigenvalue of strictly lowest modulus. For $v\in (-2,2)$, $\rho_v:\Gamma_3 \rightarrow SU(3,1)$  (see section 2.4 of \cite{Flexing}). Theorem 2.3 of \cite{Flexing} implies there exists $\epsilon>0$ such that for all $v\in (2-\epsilon,2)$, $\rho_v$ is discrete and faithful. With this notation, we can state our first main result that ``most" projective deformations of the hyperbolic structure on Vol3 give rise to holonomy representations that are strongly dense.

\begin{thm}\label{main}
	The representation $\rho_v:\Gamma_3\rightarrow SL(4,\mathbb{R})$ is strongly dense for all but a possibly countable subset of $v\in (2,\infty)$.
\end{thm} 

An immediate byproduct of our efforts is that many complex hyperbolic deformations of Vol3 also give strongly dense representations:

\begin{thm}\label{main2}
	For all $\epsilon>0$, there exists $t\in (2-\epsilon,2)$, with $\rho_t:\Gamma_3 \rightarrow SU(3,1)$ discrete and faithful, such that $\rho_t(\Gamma_3)$ is strongly dense in $SU(3,1)$.
\end{thm}

All currently known strongly dense representations of surface groups into $SL(n,\mathbb{R})$ have been \emph{Hitchin}. As shown in section 3 of \cite{ConstructingThin}, $\pi_1(Vol3)$ contains closed hyperbolic surface subgroups $\pi_1(F)<\pi_1(Vol3)$ such that for all $v\in (2,\infty)$, $\rho_v|_{\pi_1(F)}$ is not Hitchin. Since strongly dense representations remain strongly dense upon restriction to non-abelian subgroups, we have the following Corollary:

\begin{cor}
	$SL(4,\mathbb{R})$ contains non-Hitchin strongly dense surface subgroups.
\end{cor}

As previously mentioned, in \cite{BGL} the authors established the existence of a large class of groups with strongly dense embeddings into various algebraic groups. This class contains closed hyperbolic surface groups. Their results work over algebraically closed fields, but they remark that sometimes, given other technical conditions, this assumption can be relaxed. It would be interesting to see if the ideas in \cite{BGL} can be extended to our setting, or apply to other 3-manifold groups.

Our methods follow a very different strategy than that in \cite{BGL,LRW}. The arguments are more in the spirit of \cite{LR}, and originate from the curious observation that the tautological representations obtained from flexing often have specializations giving lower dimensional representations with geometric significance.

The paper is organized as follows. In section \ref{sird} we establish irreducibility results on the representation when restricted to an arbitrary non-abelian free subgroup of $\Gamma_3$. Then in section \ref{scplx}, we give some necessary background in complex hyperbolic geometry. Finally in section \ref{smain}, we prove the main theorems.

\newtheorem*{Acknowledgments}{Acknowledgments}

\begin{Acknowledgments}
	The author would like to thank Darren Long for introducing him to this field, and for many helpful discussions.
\end{Acknowledgments}

\section{Irreducibility}\label{sird}
A subgroup of $GL(4,\mathbb{R})$ is \emph{irreducible} if its action on $\mathbb{R}^4$ has no invariant subspaces. Say a subgroup of $GL(4,\mathbb{C})$ is \emph{absolutely irreducible} if its action on $\mathbb{C}^4$ has no invariant subspaces. A subgroup of $GL(4,\mathbb{R})$ is \emph{strongly irreducible} if all of its finite index subgroups are irreducible. Our first task toward the proof of our main Theorems will be to show $\rho_v$ is absolutely irreducible for all but finitely many specializations of $v$, and strongly irreducible for all but countably many specializations of $v$, when restricted to a non-abelian free subgroup of $\Gamma_3$.

Since it will be important for future arguments, we start this section with some specifics about the Minkowski model of $\mathbb{H}^3$. Recall that the isometries of the upper half space model of $\mathbb{H}^3$ is naturally identified with $PSL(2,\mathbb{C})$, while the isometries of the Minkowski model of $\mathbb{H}^3$ is given by $SO^+(3,1)$. We begin by reviewing a well known isomorphism between the two.

Write 
\begin{equation}\label{herm}
\mathcal{H}:=\left\{	\begin{bmatrix}
1 & 0  \\
0 & 1 
\end{bmatrix}, \begin{bmatrix}
1 & 0  \\
0 & -1 
\end{bmatrix},\begin{bmatrix}
0 & 1  \\
1 & 0 
\end{bmatrix},\begin{bmatrix}
0 & i  \\
-i & 0 
\end{bmatrix}\right\}.
\end{equation} 

Then span$_\mathbb{R}\mathcal{H}$ is the real vector space of $2\times 2$ Hermitian matrices. An explicit isomorphism from $\mathbb{R}^4$ to span$_\mathbb{R}\mathcal{H}$ is given by 
\[	(t,x,y,z)\mapsto \begin{bmatrix}
t+z & x+iy  \\
x-iy & t-z 
\end{bmatrix}.	\]
The function $H\mapsto -\det(H)$ gives a quadratic form on span$_\mathbb{R}\mathcal{H}$ that is preserved by the following action of $PSL(2,\mathbb{C})$: given $[A]\in PSL(2,\mathbb{C})$, let 
\[	[A].H=A^*HA	\]
where $A\in SL(2,\mathbb{C})$ is any representative of $[A]\in PSL(2,\mathbb{C})$ and $A^*$ is the Hermitian transpose of $A$. For any $A\in SL(2,\mathbb{C})$, let 
\[\tau(A):\text{span}_\mathbb{R}\mathcal{H}\rightarrow\text{span}_\mathbb{R}\mathcal{H}\]
denote the $\mathbb{R}$-linear map defined by $\tau(A)(H):= A^*HA$. Then an explicit isomorphism from $PSL(2,\mathbb{C})$ to $SO^+(3,1)$ is given by mapping each $[A]\in PSL(2,\mathbb{C})$ to the matrix representative of $\tau(A)$ with respect to $\mathcal{H}$, for any representative $A$ of $[A]$.

From this isomorphism one can deduce the following facts about hyperbolic isometries in $SO^+(3,1)$. One can consult section 2 of \cite{Computing} for more detail. A hyperbolic $g\in SO^+(3,1)$ has eigenvalues $\lambda$, $\frac{1}{\lambda}$, $e^{i\theta}$, and $e^{-i\theta}$ with $\lambda \in \mathbb{R}$ and $|\lambda|> 1$. We refer to the $\mathbb{R}$-span of any $\lambda$ and $\frac{1}{\lambda}$ eigenvector of $g$ as the attracting and repelling eigenspaces, respectively, of $g$. These are the 1-dimensional subspaces of $\mathbb{R}^4$ that contain the ray in the light cone corresponding to the attracting and repelling fixed point at infinity of $g$, respectively. We refer to the $\mathbb{C}$-span of any $e^{\pm i\theta}$ eigenvector of $g$ as a rotation eigenspace of $g$. If a hyperbolic in $SO^+(3,1)$ has a 1-eigenvalue, then it must mean $\theta=0$ and the hyperbolic is a pure translation, coming from a hyperbolic in $PSL(2,\mathbb{C})$ with real trace. 

\begin{lemma}\label{Invariant}
	Let $G=\langle \gamma_1, \gamma_2 \rangle$ be a free subgroup of $\Gamma_3$. Put $g_j=\rho_2(\gamma_j)$. If $W\subseteq \mathbb{R}^4$ is a $\rho_2(G)$-invariant subspace generated by light like vectors, then $W=\mathbb{R}^4$.
\end{lemma}

\begin{proof}
	Let $\{w_j\}$ be a basis for $W$ where each $w_j$ is a light like vector. If dim$_\mathbb{R}W=1$, then $w_1$ corresponds to a point at infinity fixed by the action of $\rho_2(G)$ on hyperbolic space. But no such point can exist because $g_1$ and $g_2$ are non-commuting hyperbolics and $\langle g_1,g_2\rangle$ is discrete.\par 
	
	If dim$_\mathbb{R}W=2$, then $W=$ span$_\mathbb{R}\{w_1,w_2\}$ is the unique geodesic with endpoints corresponding to $w_1$ and $w_2$. The $\rho_2(G)$-invariance of $W$ means this geodesic is fixed by the action of $H$ on hyperbolic space. But no such geodesic can exist because $g_1$ and $g_2$ are non-commuting hyperbolics and $\langle g_1,g_2\rangle$ is discrete.\par 
	
	If dim$_\mathbb{R}W=3$, then $\langle g_1^T, g_2^T \rangle$ has an invariant 1 dimensional subspace. That is, $g_1^T$ and $g_2^T$ have a common eigenvector. This means all the elements in the commutator subgroup $[\langle g_1^T, g_2^T \rangle,\langle g_1^T, g_2^T \rangle]$ have a 1 eigenvalue. Hence, all hyperbolics in $[\rho_2(G),\rho_2(G)]$ have a 1 eigenvalue. But hyperbolics in $SO^+(3,1)$ with a 1 eigenvalue are pure translations and correspond to hyperbolics in $PSL(2,\mathbb{C})$ with real trace. So this would imply $[\rho_2(G),\rho_2(G)]$ contains a non-elementary Fuchsian subgroup of $\Gamma_3$. But $\Gamma_3$ does not contain any non-elementary Fuchsian subgroups.\par 
	
	We conclude that dim$_\mathbb{R}W=4$ as desired.	
\end{proof}

\begin{lemma}\label{Light Basis}
	Let $G=\langle \gamma_1, \gamma_2 \rangle$ be a free subgroup of $\Gamma_3$. Then there exists a pair of hyperbolics in $\rho_2(G)$ with attracting and repelling eigenspaces that span $\mathbb{R}^4$.
\end{lemma}

\begin{proof}
	The subspace of $\mathbb{R}^4$ generated by all attracting and repelling eigenspaces of elements in $\rho_2(G)$ is a $\rho_2(G)$-invariant subspace generated by light like vectors. By Lemma \ref{Invariant}, this means we can obtain light like vectors $\{w_j\}_{j=1}^4$ forming a basis of $\mathbb{R}^4$ such that $w_j$ is an attracting eigenvector of $g_j\in \rho_2(G)$.\par 
	
	For $k\in \mathbb{N}$, consider the hyperbolics $g_1^kg_2^{-k}$ and $g_3^kg_4^{-k}$. As $k$ increases, the fixed points of $g_1^kg_2^{-k}$ converge projectively to $w_1$ and $w_2$, and the fixed points of $g_3^kg_4^{-k}$ converge projectively to $w_3$ and $w_4$. Linear independence is an open condition. Since $\{w_j\}_{j=1}^4$ is a basis for $\mathbb{R}^4$, this means for all large $k$ the attracting and repelling eigenspaces of $g_1^kg_2^{-k}$ and $g_3^kg_4^{-k}$ span $\mathbb{R}^4$.
\end{proof}

\begin{prop}\label{Irreducible}
	If $G=\langle \gamma_1, \gamma_2 \rangle$ is a free subgroup of $\Gamma_3$, then $\rho_2(G)=\langle g_1,g_2 \rangle$ is absolutely irreducible.
\end{prop}

\begin{proof}
	By Lemma \ref{Light Basis}, passing to a free subgroup if necessary, we can assume without loss of generality that the $\mathbb{C}$-span of the attracting and repelling eigenspaces of $g_1$ and $g_2$ is $\mathbb{C}^4$. Let $W\subseteq \mathbb{C}^4$ be a $\rho_2(G)$-invariant subspace. As in the proof of Lemma \ref{Invariant}, if dim$_\mathbb{C}W$ is 1 or 3, then $[\rho_2(G),\rho_2(G)]$ would contain non-elementary Fuchsian subgroups. But it is known $\rho_2(G)$ contains no such subgroup. \par 
	
	Suppose for contradiction dim$_\mathbb{C}W=2$. Since $g_1$ and $g_2$ have 4 distinct eigenvalues and $W$ is $G$-invariant, $W$ must contain two eigenspaces in $\mathbb{C}^4$ of $g_1$ and $g_2$.
	
	\begin{Claim}\label{rots}
		$W$ contains both rotation eigenspaces of $g_1$ and $g_2$.
	\end{Claim}
	
	\begin{proof}[(Proof of Claim \ref{rots})] Let $x$ be an attracting or repelling eigenvector of $g_1$ or $g_2$. Then span$_{\mathbb{R}}\{\rho_2(G)(x)\}$, the subspace of $\mathbb{R}^4$ generated by the $\rho_2(G)$-orbit of $x$, is $\rho_2(G)$-invariant. It is also generated by light like vectors. By Lemma \ref{Invariant}, there exists $\{x_j\}_{j=1}^4\subseteq \rho_2(G)(x)$ such that span$_{\mathbb{R}}\{x_j\}_{j=1}^4=\mathbb{R}^4$. Then span$_{\mathbb{C}}\{x_j\}_{j=1}^4=\mathbb{C}^4$. So we conclude $x\notin W$ because otherwise, we would have $\{x_j\}_{j=1}^4\subseteq \rho_2(G)(x)\subseteq W$, which cannot hold since we are assuming dim$_{\mathbb{C}}W=2$. Therefore, the only eigenspaces of $g_1$ and $g_2$ that $W$ can contain are their rotation eigenspaces. 		
	\end{proof}
	
	In what follows, let $J$ denote the symmetric form defining the $SO(3,1;\mathbb{R})$ version of $\mathbb{H}^3$. Let $\langle , \rangle$ denote the form on $\mathbb{C}^4$ defined by 
	\[	\langle x, w \rangle = w^*Jx.	\]
	Then $\langle , \rangle$ is a Hermitian form on $\mathbb{C}^4$ that is preserved by $\rho_2(G)$. For any subspace $V\subseteq \mathbb{C}^4$, put 
	\[	V^\perp:=\{	w\in \mathbb{C}^4 : \forall v\in V,\langle v, w \rangle =0	\}.	\]
	For any subspace $V$, $V^\perp$ is also a subspace and we have 
	\[\text{dim}_\mathbb{C}V+\text{dim}_\mathbb{C}V^\perp = 4.\]
	\begin{Claim}\label{contain}
		For $j=1,2$, if $r$ is an $e^{i\theta}$ eigenvector of $g_j$, then span$_{\mathbb{C}}\{r\}^\perp$ contains the attracting and repelling eigenspaces of $g_j$.
	\end{Claim}
	
	\begin{proof}[(Proof of Claim \ref{contain})]
		Let $x$ be a $\lambda\in \mathbb{R}$ eigenvector of $g_j$, where $|\lambda|\neq 1$. For any $k$, we have
		\begin{equation}\label{perp}
		|\langle r, x \rangle| = |\langle g_j^k(r), g_j^k(x) \rangle|= |e^{ik\theta}\lambda^k\langle r,x \rangle| = |\lambda|^k|\langle r,x \rangle|.
		\end{equation} 
		Since equation \ref{perp} holds for any $k\in \mathbb{Z}$ and $|\lambda|\neq 1$, we must have $\langle r,x \rangle=0$.
	\end{proof}	
	
	By Claim \ref{rots}, the rotation eigenspaces of $g_j$ generate $W$, for $j=1$ or $2$. In particular, $g_j$ acts on $W$ by rotations. This means Claim \ref{contain} implies that for any $w\in W$, span$_{\mathbb{C}}\{w\}^\perp$ contains the attracting and repelling eigenspaces of $g_1$ and $g_2$. But the $\mathbb{C}$-span of the attracting and repelling eigenspaces of $g_1$ and $g_2$ is $\mathbb{C}^4$, while dim$_\mathbb{C}\{w\}^\perp=3$. This gives the desired contradiction, and we conclude dim$_\mathbb{C}W=4$ as desired.
\end{proof}

\begin{prop}\label{ASirreducible}
	 
	\begin{enumerate}[(1)]
		\item If $G=\langle \gamma_1, \gamma_2 \rangle$ is a non-abelian free subgroup of $\Gamma_3$, then $\rho_v(G)$ is absolutely irreducible for all but finitely many specializations of $v$.
		\item There exists a subset $\mathcal{S}\subset [2,\infty)$ such that $[2,\infty)\setminus \mathcal{S}$ is at most countable and the image of all non-abelian rank 2 free subgroups of $\Gamma_3$ under $\rho_v$ is strongly irreducible for all $v\in \mathcal{S}$.
	\end{enumerate}
\end{prop}

\begin{proof}
	(1) Proposition \ref{Irreducible} implies $\rho_2(G)$ is absolutely irreducible. Burnside's lemma (see \cite{Burnside}) implies that the group algebra $\mathbb{C}\rho_v(G)$ contains elements  $A_1(v),\ldots,A_{16}(v)\in \mathbb{C}\rho_v(G)$ such that $\{A_1(2),\ldots,A_{16}(2)\}$ is a basis for $M(4,\mathbb{C})$. Put 
	\[	\delta(v):= \det[A_1(v),\ldots,A_{16}(v)].	\]
	All the matrix entries of elements in $\rho_v(G)$ lie in $\mathbb{Q}(v,\sqrt{v^2-4})$. This means we can regard $\delta$ as an algebraic function of the single complex parameter $v$.
	
	The failure of absolute irreducibility of $\rho_t(G)$ at the specialization $v=t$ implies either $\delta(t)=0$, or $t$ is a pole of $\delta$. The absolute irreducibility of $\rho_2(G)$ implies $\delta(2)\neq 0$, so $\delta$ has at most finitely many zeros or poles. We conclude that $\rho_v(G)$ is absolutely irreducible for all but finitely many $v$.
	
	(2) Since $\Gamma_3$ contains countably many rank 2 free groups, and all finite index subgroups of any non-abelian free $G<\Gamma_3$ contain a free group of rank 2 by a Ping-Pong argument, part (2) of this proposition follows immediately from part (1).
\end{proof}

\section{Complex Hyperbolic Geometry}\label{scplx}
Recall that for $v\in (-2,2)$, $\rho_v(\Gamma_3)$ lies in $SU(3,1)$. In this section, we give a rapid review of the necessary background in complex hyperbolic geometry for our future arguments. A standard and comprehensive reference for this material is \cite{Goldman}. 

Let $\mathbb{V}$ be the vector space $\mathbb{C}^4$ equipped with the Hermitian form defined by 
\[	\langle z, w \rangle = w^*Hz,	\]
where $H=diag(-1,1,1,1)$ is the diagonal matrix representing the Hermitian form. If $H'$ is any other Hermitian matrix with signature (3,1), then there is a $4\times 4$ matrix $C$ so that $\bar{C}^TH'C=H$. Consider the following subspaces:
\[	\mathbb{V}_-=\{	z\in \mathbb{V}:\langle z,z \rangle <0	\}, \mathbb{V}_0=\{z\in \mathbb{V}: \langle z,z \rangle=0 \}.	\] 
Complex hyperbolic space $\mathbb{CH}^3$ is defined to be the image of $\mathbb{V}_-$ under the canonical projection $\mathbb{V}\setminus\{0\}\rightarrow \mathbb{C}P^3$ onto complex projective space. The image of $\mathbb{V}_0$ under this projection gives the ideal boundary $\partial\mathbb{CH}^3$.

The holomorphic isometry group of $\mathbb{CH}^3$ is the projective unitary group $PSU(3,1)=SU(3,1)/\{\pm I, \pm iI\}$, although for convenience we will lift to the four fold cover $SU(3,1)$ when analyzing isometries. The classification of isometries of $\mathbb{CH}^3$ mirrors that of $\mathbb{H}^3$. An isometry is \emph{elliptic} if it fixes at least one point of $\mathbb{CH}^3$. It is \emph{parabolic} if it fixes exactly one point, and that point lies in $\partial \mathbb{CH}^3$. It is \emph{hyperbolic} if it fixes exactly two points, both of which lie in $\partial \mathbb{CH}^3$.

The eigenvalues of elliptic and parabolic elements all have norm 1. Hyperbolic elements have a pair of eigenvalues of the form $re^{i\theta},r^{-1}e^{i\theta}$, $r>1$, with associated one dimensional eigenspaces corresponding to its two fixed points in $\partial \mathbb{CH}^3$.

A \emph{complex hyperbolic Kleinian group} is a discrete subgroup of $SU(3,1)$. If $\Gamma<SU(3,1)$ is a complex hyperbolic Kleinian group and $x\in \mathbb{CH}^3$, then its \emph{limit set} $\Lambda_\Gamma$ is the set of accumulation points of the orbit $\Gamma(x)$ in $\partial\mathbb{CH}^3$. It is known that $\Lambda_\Gamma$ is independent of the choice of $x$. A complex hyperbolic Kleinian group $\Gamma$ is called \emph{elementary} if $|\Lambda_\Gamma|\leq 2$, and \emph{non-elementary} otherwise. 

The main result we need from complex hyperbolic geometry is the following theorem of Kim and Kim in \cite{Kim}:

\begin{thm}\label{Kim}
	Let $\Gamma<SU(3,1)$ be a non-elementary complex hyperbolic Kleinian group. Then $tr(\gamma)\in \mathbb{R}$ for all $\gamma \in \Gamma$ if and only if $\Gamma$ is conjugate to a subgroup of $SO(3,1)$ or $SU(1,1)\times SU(2)$.
\end{thm}

We conclude this section with a basic observation that will aid in our future applications of Theorem \ref{Kim}.

\begin{lemma}\label{comp}
	Let $G=\langle \gamma_1, \gamma_2 \rangle$ be a non-abelian free subgroup of $\Gamma_3$. Then there exists $\epsilon>0$ such that for all $v\in (2-\epsilon,2)$, $\rho_v(G)$ is a non-elementary complex hyperbolic Kleinian group.
\end{lemma}

\begin{proof}
	By Proposition 2.4 of \cite{Flexing}, we can choose $\epsilon>0$ such that $\rho_v$ is discrete and faithful for all $v\in (2-\epsilon,2)$, meaning $\rho_v(G)$ will be a complex hyperbolic Kleinian group for all such $v$. Since $\rho_2(\gamma_j)\in SO(3,1)$ are loxodromics of real hyperbolic space for $j=1,2$, they both have a pair of real eigenvalues $\lambda_j,\lambda_j^{-1}$ for some $\lambda_j>1$, each of which corresponding to distinct fixed points at infinity. The eigenvalues of $\rho_v(\gamma_j)$ depend continuously on $v$ so we can choose $\epsilon>0$ such that $\rho_v(\gamma_j)$ has a pair of eigenvalues that are not of unit norm, for all $v\in (2-\epsilon,2)$. That is we can choose $\epsilon>0$ so that for all $v\in (2-\epsilon,2)$, $\rho_v(\gamma_j)\in SU(3,1)$ will be a pair of non-commuting complex hyperbolics. This means $\rho_v(\gamma_1)$ and $\rho_v(\gamma_2)$ will have disjoint fixed point sets. Since $\Lambda_{\rho_v(G)}$ contains the attracting and repelling fixed points of $\rho_v(\gamma_j)$, we conclude $\rho_v(G)$ is also non-elementary for all such $v$.      
\end{proof}

\section{Proof of Theorems}\label{smain}

Our first goal of this section will be to show that if $G<\Gamma_3$ is a free group of rank $2$ and $v\in (2,\infty)$ is such that $\rho_v(G)$ is strongly irreducible, then $\rho_v(G)$ is Zariski dense in $SL(4,\mathbb{R})$, except for possibly finitely many exceptional specializations of $v$. Once this is established, we then show $\rho_v(\Gamma_3)$ is strongly dense in $SL(4,\mathbb{R})$, with possibly countably many exceptional specializations of $v\in (2,\infty)$. Analogous results for $SU(3,1)$ deformations will then be clear once our techniques are established. 

Our analysis of Zariski closures is based on the following result of Benoist, whose restatement can be found in Proposition 3.6 of \cite{semisimple}:

\begin{thm}\label{Benoist}
	Suppose $\Gamma$ is a strongly irreducible subgroup of $SL(n+1,\mathbb{R})$ which preserves a properly convex $\Omega\subset \mathbb{RP}^n$. Let $\mathcal{G}$ be the Zariski closure of $\Gamma$. Then $\mathcal{G}$ is a Zariski-connected real semi-simple Lie group.
\end{thm}

Recall that for $v\in (2,\infty)$, $\rho_v:\Gamma_3\rightarrow SL(4,\mathbb{R})$ is a discrete faithful representation that is the holonomy of a strictly convex projective structure on the orbifold. This means Theorem \ref{Benoist} and Proposition \ref{ASirreducible} implies that for all $v\in (2,\infty)$, except for possibly finitely many exceptional specializations of $v$, $\rho_v(G)$ has Zariski closure a semi-simple Lie subgroup of $SL(4,\mathbb{R})$. All such subgroups have been classified, one can check \cite{subalgebras} for a reference. So to show a given strongly irreducible $\rho_v(G)$ is Zariski dense in $SL(4,\mathbb{R})$, we will show no other Lie subgroups of $SL(4,\mathbb{R})$ can be a possible Zariski closure.

Since $G$ is a free group of rank 2 and $\rho_v$ is discrete, faithful, and irreducible, most Lie groups can be easily ruled out as possible Zariski closures by general algebraic reasons. The main obstruction to Zariski density in our specific context is the possibility that $\rho_v(G)$ preserves either a symmetric or skew-symmetric form. This is handled by the following series of Lemmas and Proposition. 

\begin{lemma}\label{split}
	Let $G$ be a non-abelian free subgroup of $\Gamma_3$. Then there is no non-zero real symmetric form preserved by $\rho_{i\sqrt{2}}(G)$.
\end{lemma}

\begin{proof}
	Recall that $\Gamma_3$ is the orbifold fundamental group of an orbifold that is four-fold covered by Vol3. This means, by passing to a finite index subgroup if necessary, we can assume without loss of generality that $G<\pi_1(\text{Vol3})$. 
	
	At the specialization $v=i\sqrt{2}$, one checks that the representation is reducible, and that we can conjugate so that 
	\[
	\rho_{i\sqrt{2}}(\gamma)=
	\begin{bmatrix}
	\overline{\rho_\infty(\gamma)} & 0  \\
	\star & \rho_\infty(\gamma)
	\end{bmatrix},
	\]
 where $\rho_{\infty}:\pi_1(\text{Vol3})\rightarrow SL(2,\mathbb{C})$ is a lift of the holonomy representation of $\pi_1(\text{Vol3})$ into $PSL(2,\mathbb{C})$ coming from the hyperbolic structure.
 
 Indeed, letting 
 \[M=\begin{bmatrix}
 0 & 0 & 1 & 0 \\
 1 & 0 & 0 & 0 \\
 0 & 0 & 0 & 1 \\
 0 & 1 & 0 & 0 \\ 
 \end{bmatrix},\]
 
 direct calculation yields 
 
 \[
 	M^{-1}\rho_{i\sqrt{2}}(u)M=\begin{bmatrix}
 	1 & 0 & 0 & 0 \\
 	0 & -i & 0 & 0 \\
 	0 & 0 & 1 & 0 \\
 	0 & 1 & 0 & i \\ 
 	\end{bmatrix} 
 \]
 and 
 \[
 M^{-1}\rho_{i\sqrt{2}}(c)M=\begin{bmatrix}
 \frac{1}{4}(i\sqrt{2}-\sqrt{6}) & \frac{1}{8}(-6-2i\sqrt{3}) & 0 & 0 \\
 -1 & \frac{1}{4}(-i\sqrt{2}+\sqrt{6}) & 0 & 0 \\
 0 & 0 & \frac{1}{4}(i\sqrt{2}+\sqrt{6}) & \frac{1}{8}(6-2i\sqrt{3}) \\
 0 & 0 & 1 & \frac{1}{4}(-i\sqrt{2}-\sqrt{6}) \\ 
 \end{bmatrix}.  
 \]
 Let
	\[
	J=
	\begin{bmatrix}
	J_1 & J_2  \\
	J_3 & J_4 
	\end{bmatrix},
	\]
	be a symmetric form preserved by $\rho_{i\sqrt{2}}(G)$. That is, $\rho_{i\sqrt{2}}(\gamma)^TJ\rho_{i\sqrt{2}}(\gamma)=J$ for all $\gamma\in G$. This means the system of equations
	\begin{equation}\label{sym}
	\rho_{\infty}(\gamma)^TJ_4\rho_{\infty}(\gamma)=J_4
	\end{equation}
	must hold for all $\gamma \in G$.\par 
	The system (\ref{sym}) can only be satisfied if $J_4$ is a real symmetric form preserved by $\rho_\infty(G)$. But $\rho_\infty(G)$ is Zariski dense in $SL(2,\mathbb{C})$ because $\pi_1(\text{Vol3})$ contains no non-elementary Fuchsian subgroups. Since the system of equations (\ref{sym}) imply that the algebraic subgroup 
	\[	\{A\in SL(2,\mathbb{C}): A^TJ_4A=J_4\}	\]
	of $SL(2,\mathbb{C})$ must contain $\rho_\infty(G)$, the Zariski density of $\rho_{\infty}(G)$ in $SL(2,\mathbb{C})$ means that every $A\in SL(2,\mathbb{C})$ must satisfy the system (\ref{sym}). Therefore, $J_4$ must be symplectic. Since $J_4$ must also be symmetric, we conclude that $J_4=0$.\par 
	
	Since $J_4=0$, then the system (\ref{sym}) holds only if 
	\[	\rho_\infty(\gamma)^*J_3\rho_\infty(\gamma)=J_3	\]
	for all $\gamma \in G$. 
	
	\begin{Claim}\label{complexified}
		There is no non-zero $M\in M(2,\mathbb{C})$ such that $\rho_{\infty}^*(\gamma)M\rho_{\infty}(\gamma)=M$ for all $\gamma\in G$.
	\end{Claim}

	\begin{proof}[(proof of Claim \ref{complexified})]
		First, we recall the following notation from Section \ref{sird}. As before, $\mathcal{H}$ denotes the basis, fixed in line (\ref{herm}), of the $2\times 2$ Hermitian matrices as a real vector space. For any $A\in SL(2,\mathbb{C})$, 
		\[\tau(A):\text{span}_\mathbb{R}\mathcal{H}\rightarrow\text{span}_\mathbb{R}\mathcal{H}\] is the $\mathbb{R}$-linear map defined by $\tau(A)(H)=A^*HA$. An isomorphism from $PSL(2,\mathbb{C})$ to $SO^+(3,1)$ is given by sending each $[A]\in PSL(2,\mathbb{C})$ to the matrix representation of $\tau(A)$ with respect to $\mathcal{H}$ for any representative $A$ of $[A]\in PSL(2,\mathbb{C})$. That is, in the notation of the proof of this Proposition, we have that $\rho_2(\gamma)$ is the matrix representation of the $\mathbb{R}$ linear map $\tau(\rho_{\infty}(\gamma))$. Therefore, $\rho_2(\gamma)$ is also the matrix representation of the complexification 
		\[\tau(A)_\mathbb{C}:\text{span}_\mathbb{C}\mathcal{H}\rightarrow\text{span}_\mathbb{C}\mathcal{H}.\]
		Since span$_\mathbb{C}\mathcal{H}=M(2,\mathbb{C})$, the absolute irreducibility of $\rho_2(G)$ (see Proposition \ref{Irreducible}) gives the desired claim, as any $M\in M(2,\mathbb{C})$ satisfying $\rho_{\infty}^*(\gamma)M\rho_{\infty}(\gamma)=M$ for all $\gamma\in G$ would correspond to a vector in $\mathbb{C}^4$ fixed by $\rho_2(G)$.  
	\end{proof}
	Claim \ref{complexified} implies $J_3=0$. Hence $J_2=J_3^T=0$. Then similar reasoning will show $J_1=0$.
\end{proof}

\begin{lemma}\label{nosym}
	Let $G$ be a non-abelian rank 2 free subgroup of $\Gamma_3$. There is no interval $I\subset \mathbb{R}$ containing 2 with non-empty interior such that $\rho_v(G)$ preserves a symmetric form for all $v\in I$.
\end{lemma}
\begin{proof}
	Suppose for contradiction there exists such an interval $I$. For all $v\in I$, let $J_v$ denote the non-zero symmetric form preserved by $\rho_v(G)$. Up to scaling, $J_v$ is the unique form satisfying the system of equations 
	\begin{equation}\label{preserve}
	\rho_v(\gamma)^TJ_v\rho_v(\gamma)=J_v
	\end{equation}
	for all $v\in I$ and $\gamma\in G$. Therefore, we can regard the entries of $J_v$ as functions with the following shape: 
	\begin{equation}\label{entries}
	p(v)+q(v)\sqrt{v^2-4}+r(v)\sqrt{v^2+8}+s(v)\sqrt{v^2-4}\sqrt{v^2+8},
	\end{equation} 
	where we have scaled $J_v$ so that for all the entries, the functions $p,q,r,s$ as in equation \ref{entries} are in $\mathbb{Q}[v]$ with no common factors.\par 
	
	Lemma \ref{split} implies that equation \ref{entries} vanishes at $v=i\sqrt{2}$. That is, 
	\begin{equation}\label{vanish}
	p(i\sqrt{2}) + q(i\sqrt{2})\sqrt{-6} + r(i\sqrt{2})\sqrt{6} + 6is(i\sqrt{2})=0.
	\end{equation}
	Applying the Galois automorphism $\sqrt{6} \mapsto -\sqrt{6}$ to equation \ref{vanish} yields 
	\begin{equation}\label{conj1}
	p(i\sqrt{2}) + q(i\sqrt{2})\sqrt{-6} - r(i\sqrt{2})\sqrt{6} + 6is(i\sqrt{2})=0.
	\end{equation}
	Combining equations \ref{vanish} and \ref{conj1} gives
	\begin{equation}\label{cancelled}
	p(i\sqrt{2}) + q(i\sqrt{2})\sqrt{-6} + 6is(i\sqrt{2})=0.
	\end{equation}
	That is, $(v^2+2)$ divides $r(v)$. Similarly, applying the Galois automorphism $\sqrt{-6}\mapsto -\sqrt{-6}$ to equation \ref{cancelled} will imply $(v^2+2)$ divides $q(v)$. Then applying the Galois automorphism $6i \mapsto -6i$ will show $(v^2+2)$ divides $s(v)$. Since this implies $q(i\sqrt{2})=r(i\sqrt{2})=s(i\sqrt{2})=0$, then equation \ref{vanish} implies $(v^2-2)$ divides $p(v)$. \par 
	
	This means in all the entries of $J_v$, the polynomials $p,q,r,$ and $s$ all have the common factor $(v^2+2)$. This establishes the desired contradiction.
\end{proof}

\begin{prop}\label{nopal}
	Let $G$ be a non-abelian rank 2 free subgroup of $\Gamma_3$. There exists $\gamma\in G$ such that the characteristic polynomial $\chi_{\rho_v(\gamma)}$ is not palindromic for all but finitely many $v\in \mathbb{C}$. 
\end{prop}

\begin{proof}
	In \cite{Flexing}, it is shown that the characteristic polynomial $\chi_{\rho_v(\gamma)}(Q)$ of $\rho_v(\gamma)$ has the following shape:
	\[	1 - (p_\gamma(v)- q_\gamma(v)\sqrt{v^2-4})Q + r_\gamma(v)Q^2 -(p_\gamma(v)+ q_\gamma(v)\sqrt{v^2-4})Q^3+Q^4.	\]
	Here, $p_\gamma,q_\gamma,r_\gamma\in \mathbb{Q}[v]$. This means $\chi_{\rho_v(\gamma)}$ is palindromic if and only if either $v\in \{\pm 2\}$ or $q_\gamma(v)\equiv0$.\par 
	
	Assume for contradiction $\chi_{\rho_v(\gamma)}$ is palindromic for infinitely many $v\in \mathbb{C}$ and all $\gamma\in G$. This means $q_\gamma\equiv 0$, and hence $tr\rho_v(\gamma)=-p_\gamma(v)$ for all $\gamma \in G$. In particular, we have $tr\rho_v(\gamma)\in \mathbb{R}$ for all $\gamma\in G$ and $v\in (-2,2)$.\par 
	
	Recall that the image of $\rho_v$ lies in $SU(3,1)$ for all $v\in (-2,2)$. By Proposition \ref{comp}, there exists $\epsilon>0$ such that $\rho_v(G)$ is a non-elementary complex hyperbolic Kleinian group for all $v\in (2-\epsilon,2)$. By Lemma \ref{ASirreducible}, we can take $\epsilon$ so that $\rho_v(G)$ is also absolutely irreducible. Since $\rho_v(G)$ is an irreducible complex hyperbolic Kleinian group such that $tr\rho_v(\gamma)\in \mathbb{R}$ for all $\gamma\in G$, Theorem \ref{Kim} implies $\rho_v(G)$ can be conjugated into $SO(3,1)$ for all $v\in (2-\epsilon,2)$. This contradicts Lemma \ref{nosym}.	
\end{proof}
We now put together the previous results to prove Theorems \ref{main} and \ref{main2}.
\begin{proof}[\textbf{Proof of Theorem \ref{main}}]
	By Proposition \ref{ASirreducible}, we can obtain a set $\mathcal{S}\subset [2,\infty)$ such that $[2,\infty)\setminus \mathcal{S}$ is at most countable and $\rho_v(G)$ is strongly irreducible for all $v\in \mathcal{S}$ and all non-abelian rank 2 free subgroups $G<\Gamma_3$. Fix $G$ a non-abelian rank 2 free subgroup of $\Gamma_3$. We first show that for all but a possibly finite subset of $v\in \mathcal{S}$, $\rho_v(G)$ is Zariski dense in $SL(4,\mathbb{R})$.
	
	Let $\overline{\rho_v(G)}$ denote the Zariski closure of $\rho_v(G)$ in $SL(4,\mathbb{R})$ and $\overline{\rho_v(G)}^0$ denote its identity component. By the strong irreducibility of $\rho_v(G)$ for $v\in \mathcal{S}$, Theorem \ref{Benoist} implies $\overline{\rho_v(G)}$ is a semi-simple Lie subgroup of $SL(4,\mathbb{R})$. But $\rho_v$ is also discrete and faithful, meaning it discretely embeds a free group of rank 2 into $\overline{\rho_v(G)}$. Hence, $\overline{\rho_v(G)}$ cannot be compact or solvable. By the classification of Lie sub-algebras of $sl(4,\mathbb{R})$ (see \cite{subalgebras}), if we can show $\overline{\rho_v(G)}^0$ does not preserve any form, then we can conclude $\rho_v(G)$ is Zariski dense in $SL(4,\mathbb{R})$. 
	
	By Proposition \ref{nopal}, there exists $\rho_v(\gamma)\in \rho_v(G)$ such that its characteristic polynomial $\chi_{\rho_v(\gamma)}$ is not palindromic for all but finitely many specializations of $v$. 
	
	\begin{Claim}\label{powers}
		If $v\in \mathcal{S}$ is such that $\chi_{\rho_v(\gamma)}$ is not palindromic, then $\rho_v(G)$ is Zariski dense in $SL(4,\mathbb{R})$.
	\end{Claim}
	
	\begin{proof}[(proof of Claim \ref{powers})]
		Recall that $\rho_v(\gamma)$ is biproximal, meaning it has a unique real eigenvalue of largest modulus, and a unique real eigenvalue of smallest modulus. Let 
		\[	|\lambda_1|>|\lambda_2|\geq |\lambda_3|>|\lambda_4|	\]
		denote the eigenvalues of $\rho_v(\gamma)$.
		Also recall that $\chi_{\rho_v(\gamma)}$ being non-palindromic is equivalent to the eigenvalues of $\rho_v(\gamma)$ not being multiplicatively symmetric about 1. These facts, along with $det\rho_v(\gamma)=1$ and $\lambda_1,\lambda_4\in \mathbb{R}$, implies that $|\lambda_2|^n|\lambda_3|^n\neq 1$ for all $n\in \mathbb{N}$. That is, $\rho_v^n(\gamma)$ has non-palindromic characteristic polynomial for all $n$ because the eigenvalues of $\rho_v^n(\gamma)$ can never be multiplicatively symmetric about 1. In particular, no power of $\rho_v(\gamma)$ can preserve any form. Since $\overline{\rho_v(G)}^0$ is a finite index subgroup of $\overline{\rho_v(G)}$, this implies $\rho_v(G)$ is Zariski dense in $SL(4,\mathbb{R})$.		
	\end{proof}
	
	Claim \ref{powers} implies that every non-abelian rank 2 free subgroup of $\Gamma_3$ has image under $\rho_v$ a Zariski dense subgroup of $SL(4,\mathbb{R})$, for all but possibly countably many $v\in \mathcal{S}$. A Ping-Pong argument shows that every pair of non-commuting elements in $\Gamma_3$ generates a subgroup that contains a non-abelian rank 2 free subgroup. Hence, we can conclude that $\rho_v:\Gamma_3 \rightarrow SL(4,\mathbb{R})$ is strongly dense for all but possibly countably many $v\in \mathcal{S}$.	
\end{proof}

\begin{proof}[\textbf{Proof of Theorem \ref{main2}}]
		By theorem 2.3 of \cite{Flexing}, we can take $\epsilon>0$ such that for all $v\in (2-\epsilon,2)$, $\rho_v$ is discrete and faithful. Proposition \ref{ASirreducible} implies there exists $t\in (2-\epsilon,2)$ such that the image of every non-abelian rank 2 free subgroup of $\Gamma_3$ under $\rho_t$ is absolutely irreducible. By Proposition \ref{nopal}, we can choose this $t\in (2-\epsilon,2)$ so that every non-abelian rank 2 free subgroup of $\Gamma_3$ has an element whose image under $\rho_t$ has characteristic polynomial that is not palindromic.
		
		A Ping-Pong argument shows that every pair of non-commuting elements in $\Gamma_3$ generates a subgroup that contains a non-abelian rank 2 free group. So we will get our desired result if we can show that every non-abelian rank 2 free subgroup of $\Gamma_3$ has image under $\rho_t$ a Zariski dense subgroup of $SU(3,1)$. 
		
		Let $G<\Gamma_3$ be a non-abelian rank 2 free group. Let $\overline{\rho_t(G)}$ denote the Zariski closure of $\rho_t(G)$ in $SU(3,1)$. Then $\overline{\rho_t(G)}$ is a closed subgroup of $SU(3,1)$ with its classical topology. Theorem 4.4.2 of \cite{closedcplx} implies that either $\overline{\rho_t(G)}=SU(3,1)$, or $\overline{\rho_t(G)}$ preserves a fixed point in $\partial \mathbb{CH}^3$, or $\overline{\rho_t(G)}$ preserves a totally geodesic submanifold of $\mathbb{CH}^3$. Since $\rho_t(G)$ is absolutely irreducible by our construction of $t$, $\overline{\rho_t(G)}$ cannot have a global fixed point. If $\overline{\rho_t(G)}$ preserves a totally geodesic submanifold of $\mathbb{CH}^3$, then Proposition 2.5.1 of \cite{closedcplx} on the classification of totally geodesic subspaces implies that $\overline{\rho_t(G)}$ is contained in a copy of $SO(3,1)$. But $\rho_t(G)$ contains an element with non-palindromic characteristic polynomial, and hence cannot preserve a symmetric form. We conclude $\overline{\rho_t(G)}=SU(3,1)$, and our desired result follows.	
\end{proof}


\begin{thebibliography}{9}	
	\bibitem{Ballas}
	Ballas, S. Finite Volume Properly-Convex Deformations of the Figure-Eight Knot Geom. Dedicata 178
	(2015), 49-73
	
	\bibitem{Benoist1}
	Benoist, Y. “Convexes divisibles. II.” Duke Mathematical Journal 120, no. 1 (2003): 97–120.
	
	\bibitem{Benoist2}
	Benoist, Y. “Convexes divisibles. III.” Annales Scientifiques de l'Ecole Normale Superieure
	38, no. 5 (2005): 793–832.
	
	\bibitem{BGL}
	Breuillard, Emmanuel \& Guralnick, Robert \& Larsen, Michael. (2022). Strongly dense free subgroups of semisimple algebraic groups II. 10.48550/arXiv.2212.08114. 
	
	\bibitem{Arithmetic}
	C Maclachlan, A W Reid, The arithmetic of hyperbolic 3–manifolds, Graduate Texts
	in Mathematics 219, Springer, New York (2003)
	
	\bibitem{Computing}
	Cooper, D., D. D. Long, and M. Thistlethwaite. “Computing varieties of representations of
	hyperbolic 3-manifolds into SL(4, R).” Experimental Mathematics 15, no. 3 (2006): 291–305.
	
	\bibitem{Flexing}
	D. Cooper, D. D. Long, and M. B. Thistlethwaite. Flexing closed hyperbolic manifolds. Geom. Topol., 11:2413–
	2440, 2007.
	
	\bibitem{Thin3}
	D. D. Long and A. W. Reid, Constructing thin subgroups in SL.4; R/, Int. Math. Res. Not. (2013), rns281.
	
	\bibitem{ConstructingThin}
	D. D. Long and A. W. Reid, Constructing thin groups, Thin groups and superstrong approximation, Math. Sci. Res. Inst. Publ., vol. 61, Cambridge Univ. Press, 2014, pp. 151–166.
	
	\bibitem{LR}
	D. D. Long and A. W. Reid. Strongly dense representations of surface groups, 2021. Preprint
	
	\bibitem{Tao}
	E. Breuillard, B. Green, R. Guralnick and T. Tao, Strongly dense free subgroups of semisimple
	algebraic groups, Israel J. Math. 192 (2012), 3
	
	\bibitem{subalgebras}
	Ghanam, Ryad \& Thompson, Gerard. (2016). Nonsolvable Subalgebras of $\mathfrak{gl}(4,\mathbb{R})$. Journal of Mathematics. 2016. 1-17. 10.1155/2016/2570147. 
	
	\bibitem{semisimple}
	Harrison Bray. Geodesic flow of nonstrictly convex Hilbert geometries. To appear, Annales de l’institute Fourier, 2020.
	
	\bibitem{BIRS}
	https://www.math.ucdavis.edu/~kapovich/EPR/birs-problem-list.pdf
	
	\bibitem{Burnside}
	I. Halperin, P. Rosenthal, Burnside’s theorem on algebras of matrices, Amer. Math. Monthly 87 (1980) 810.
	
	\bibitem{Kim}
	J Kim, S Kim, A characterization of complex hyperbolic Kleinian groups in dimension 3 with trace fields contained in $\mathbb{R}$, Linear Algebra Appl. 455 (2014) 107 MR3217402
	
	
	\bibitem{Kozul}
	Koszul, J.-L. “Deformation des connexions localement plates.”  Annales de l'Institut Fourier
	18, no. 1 (1968): 103–14.
	
	\bibitem{LRW}
	Long, Darren D. and Reid, Alan W. and Wolff, Maxime (2022) Most Hitchin representations are strongly dense. To appear in Michigan Journal of Math. 
	
	\bibitem{closedcplx}
	S. Chen, L. Greenberg; Hyperbolic spaces, in Contributions to Analysis. Academic Press,
	New York (1974), 49–87.
	
	\bibitem{Goldman}
	W.M. Goldman, Complex Hyperbolic Geometry, Oxford University Press, 1999
					
\end{thebibliography}
\end{document}